\pdfoutput=1
\RequirePackage{ifpdf}
\ifpdf % We are running pdfTeX in pdf mode
\documentclass[pdftex]{sigma}
\else
\documentclass{sigma}
\fi

\numberwithin{equation}{section}

\newtheorem{Theorem}{Theorem}[section]
\newtheorem{Lemma}[Theorem]{Lemma}
\newtheorem{Proposition}[Theorem]{Proposition}

\DeclareMathOperator{\Bal}{Bal}
\DeclareMathOperator{\supp}{supp}

\begin{document}

%\allowdisplaybreaks

\renewcommand{\thefootnote}{$\star$}

\newcommand{\arXivNumber}{1604.03070}

\renewcommand{\PaperNumber}{065}

\FirstPageHeading

\ShortArticleName{A Vector Equilibrium Problem for Muttalib--Borodin Biorthogonal Ensembles}

\ArticleName{A Vector Equilibrium Problem\\ for Muttalib--Borodin Biorthogonal Ensembles\footnote{This paper is a~contribution to the Special Issue
on Asymptotics and Universality in Random Matrices, Random Growth Processes, Integrable Systems and Statistical Physics in honor of Percy Deift and Craig Tracy. The full collection is available at \href{http://www.emis.de/journals/SIGMA/Deift-Tracy.html}{http://www.emis.de/journals/SIGMA/Deift-Tracy.html}}}

\Author{Arno B.J.~KUIJLAARS}

\AuthorNameForHeading{A.B.J.~Kuijlaars}

\Address{Katholieke Universiteit Leuven, Department of Mathematics,\\ Celestijnenlaan 200B box 2400, BE-3001 Leuven, Belgium}
\Email{\href{mailto:arno.kuijlaars@kuleuven.be}{arno.kuijlaars@kuleuven.be}}
\URLaddress{\url{http://perswww.kuleuven.be/~u0017946/}}

\ArticleDates{Received April 12, 2016, in f\/inal form July 03, 2016; Published online July 05, 2016}

\Abstract{The Muttalib--Borodin biorthogonal ensemble is a joint density function for $n$ particles on the positive real line that depends on a~parameter~$\theta$. There is an equilibrium problem that describes the large~$n$ 	behavior. We show that for rational values of $\theta$ there is an equivalent vector equilibrium problem.}


\Classification{31A05; 60B20}

\begin{flushright}
\begin{minipage}{70mm}
\it Dedicated to Percy Deift and Craig Tracy\\ on the occasion of their 70th birthdays
\end{minipage}
\end{flushright}

\renewcommand{\thefootnote}{\arabic{footnote}}
\setcounter{footnote}{0}

\section{Introduction and statement of results}\label{section1}

\subsection{The Muttalib--Borodin ensemble}\label{section1.1}

The Muttalib--Borodin biorthogonal ensemble is the following probability density function for $n$ particles on the half line $[0, \infty)$
\begin{gather} \label{eq:Bor1}
	\frac{1}{Z_n} \prod_{j < k} (x_k-x_j) \prod_{j < k} \big(x_k^{\theta} - x_j^{\theta}\big) \prod_{j=1}^n e^{-nV(x_j)},
	\qquad x_j \geq 0, \end{gather}
with $\theta > 0$ and with an $n$-dependent weight function $w(x) = e^{-nV(x)}$ having enough decay at inf\/inity. The model is named after Muttalib~\cite{Mut95} who introduced it as a simplif\/ied model for disordered conductors in the metallic regime, and Borodin~\cite{Bor99} who obtained profound mathematical results, in particular for Laguerre and Jacobi weights.

The model has attracted considerable attention in recent years. Random matrix models whose eigenvalues (or singular values) have the distribution~\eqref{eq:Bor1} were recently given in~\cite{Che14, ForWan15}. The model is also related to products of random matrices~\cite{ForWan15, KuiSti14}.

In the large $n$ limit, the particles have an almost sure limiting measure $\mu^*$ which is the minimizer of
\begin{gather} \label{eq:CRfunctional}
	\frac{1}{2} \iint \log \frac{1}{|x-y|} d\mu(x) d\mu(y) +
	\frac{1}{2} \iint \log \frac{1}{|x^{\theta} - y^{\theta}|} d\mu(x) d\mu(y)
		+ \int V(x) d\mu(x) \end{gather}
among all probability measures $\mu$ on $[0,\infty)$. This follows from large deviation results for \eqref{eq:Bor1} and related models that were studied in \cite{BlLeToWi15, But16, EiSoSt11}. For $\theta = 1$, the functional~\eqref{eq:CRfunctional} reduces to the usual energy in the presence of an external f\/ield \cite{SaTo97}.

The minimizer for \eqref{eq:CRfunctional} was studied in detail by Claeys and Romano~\cite{ClaRom13}. They found suf\/f\/icient conditions for the minimizer to be supported on an interval $[0,a]$ for some $a > 0$. Forrester and co-authors~\cite{ForLiu14, FoLiZJ15} analyzed the equilibrium problem for~\eqref{eq:CRfunctional} with special potentials, and found expressions for the minimizers as Fuss--Catalan and Raney distributions, see also~\cite{NeuVAs16}.

It is the aim of this paper to show that for rational values of $\theta$, say $\theta = q/r$ with $q,r \in \mathbb N$, there is an equivalent minimization problem for a vector of $q+r-1$ measures. We expect that the vector equilibrium problem will be useful for subsequent asymptotic analysis. The special role of rational $\theta$ also appeared in the already mentioned work~\cite{ClaRom13}. This paper gives f\/inite term recurrence relations for the biorthogonal polynomials associated with~\eqref{eq:Bor1}, as well as a~Christof\/fel--Darboux formula for the correlation kernel are given, but only for rational~$\theta$.

In order to state our results we introduce the logarithmic energy
\begin{gather} \label{eq:logenergy}
	I(\mu) = \iint \log \frac{1}{|x-y|} d \mu(x) d\mu(y)
	\end{gather}
	of a measure $\mu$ and the mutual energy
\begin{gather} \label{eq:mutualenergy}
	I(\mu,\nu) = \iint \log \frac{1}{|x-y|} d \mu(x) d\nu(y)
	\end{gather}
of two measures $\mu$ and $\nu$. Throughout we use for $j \in \mathbb Z$,
\begin{gather} \label{eq:Deltaj}
	\Delta_j = \begin{cases} [0,\infty), & \text{if $j$ is even}, \\
		(-\infty,0], & \text{if $j$ is odd}.
		\end{cases}
\end{gather}

\subsection[Result for the case $\theta = 1/r$]{Result for the case $\boldsymbol{\theta = 1/r}$}\label{section1.2}

We f\/irst state the result for the case $\theta = 1/r$ with $r \in \mathbb N$.

\begin{Theorem} \label{thm11}
	Let $V\colon [0,\infty) \to \mathbb R$ be continuous and
	\begin{gather*} %\label{eq:Vgrowth}
		\lim_{x \to \infty} \big(V(x) - \log\big(1+x^2\big) \big) = + \infty.
	\end{gather*}
	Then there is a unique vector $(\mu_0^*, \mu_1^*, \ldots, \mu_{r-1}^*)$ of~$r$ measures that minimizes the energy functional
	\begin{gather} \label{eq:VEproblem}
	\sum_{j=0}^{r-1} I(\mu_j) - \sum_{j=0}^{r-2} I(\mu_j, \mu_{j+1})
	+ \int V d\mu_0
	\end{gather}
	among all vectors satisfying for every $j=0,\ldots, r-1$,
	\begin{enumerate}\itemsep=0pt
	\item[$i)$] $\supp(\mu_j) \subset \Delta_j$,
	\item[$ii)$] $\mu_j(\Delta_j) = 1 - \frac{j}{r}$, and
	\item[$iii)$] $I(\mu_j) < +\infty$.
	\end{enumerate}
	
	The measure $\mu_0^*$ has compact support, and it is the 	unique minimizer of the functional~\eqref{eq:CRfunctional} with $\theta = 1/r$ 	among probability measures on $[0,\infty)$.	
\end{Theorem}

The minimization problem for the energy functional~\eqref{eq:VEproblem} is an example of a weakly admissible vector equilibrium problem in the sense of Hardy and Kuijlaars, see also below. The other measures $\mu_1^*, \ldots, \mu_{r-1}^*$ have full unbounded support, $\supp(\mu_j^*) = \Delta_j$ for $j=1, \ldots, r-1$. In fact $ \mu_j^* = \frac{1}{2} \Bal\big(\mu_{j-1}^* + \mu_{j+1}^*, \Delta_j \big)$, where~$\Bal$ denotes the balayage onto~$\Delta_j$, see below as well.

\subsection[Result for rational $\theta$]{Result for rational $\boldsymbol{\theta}$}\label{section1.3}

For general rational $\theta = q/r$ with $q,r \in \mathbb N$, we f\/irst make the change of variables $x \mapsto x^{1/q}$, $y \mapsto y^{1/q}$ in the energy functional~\eqref{eq:CRfunctional} to obtain
\begin{gather} \label{eq:CRfunctional2}
\frac{1}{2} \iint\! \log \frac{1}{|x^{1/q}-y^{1/q}|} d\nu(x) d\nu(y) +
\frac{1}{2} \iint\! \log \frac{1}{|x^{1/r} - y^{1/r}|} d\nu(x) d\nu(y)
+ \int \! \widehat{V}(x) d\nu(x),\!\!\!\! \end{gather}
where $d\nu(x) = d\mu(x^{1/q})$, and $\widehat{V}(x) = V(x^{1/q})$. Note that $q$ and $r$ play a symmetric role in the energy functional~\eqref{eq:CRfunctional2}.

\begin{Theorem} \label{thm12}
Let $\widehat{V}\colon [0,\infty) \to \mathbb R$ be continuous and
	\begin{gather} \label{eq:tildeVgrowth}
	\lim_{x \to \infty} \big(\widehat{V}(x) - \log\big(1+x^2\big) \big) = + \infty.
	\end{gather}
Then there is a unique vector $(\nu_{-q+1}^*, \ldots, \nu_{-1}^*, \nu_0^*, \nu_1^*, \ldots,	\nu_{r-1}^*)$ of $q+r-1$ measures that minimizes the energy functional
	\begin{gather} \label{eq:VEproblem2}
	\sum_{j=-q+1}^{r-1} I(\nu_j) - \sum_{j=-q+1}^{r-2} I(\nu_j, \nu_{j+1})
	+ \int \tilde{V} d\nu_0
	\end{gather}
	among all vectors satisfying for every $j=-q+1,\ldots, r-1$,
	\begin{enumerate}\itemsep=0pt
		\item[$i)$] $\supp(\nu_j) \subset \Delta_j$,
		\item[$ii)$] $\nu_j(\Delta_j) = m_j :=
		\begin{cases} 1 - \frac{j}{r} & \text{if }j \geq 0, \vspace{1mm}\\ 1 - \frac{|j|}{q} & \text{if } j \leq 0, \end{cases}$
		\item[$iii)$] $I(\mu_j) < +\infty$.
	\end{enumerate}
	
The measure $\nu_0^*$ has compact support and it is the unique minimizer of the functional~\eqref{eq:CRfunctional2} among probability measures on $[0,\infty)$, and $d\mu_0^*(x) = d\nu_0^*(x^{q})$ 	is the unique minimizer of~\eqref{eq:CRfunctional} with $\theta = q/r$.
\end{Theorem}

As in Theorem~\ref{thm11} the other measures $\nu_j^*$ from Theorem~\ref{thm12} have full unbounded support: $\supp(\nu_j^*) = \Delta_j$ if $j \neq 0$.

The energy functional \eqref{eq:VEproblem2} with $\hat{V} = 0$ and the normalizations $m_j$ as in condition ii) appeared in \cite[Theorem~2.3]{DuKu08} where it describes the limiting eigenvalue distribution of banded Toeplitz matrices. The supports $\Delta_j$ of the measures, however, are more general curves in that case.

Theorem \ref{thm11} is the special case $q=1$ of Theorem \ref{thm12} and it is enough to prove the latter theorem. However, for sake of exposition we chose to state Theorem~\ref{thm11} separately as well.

\section{Preliminaries}\label{section2}

\subsection{Vector equilibrium problems}\label{section2.1}

The unique existence of a minimizing vector of measures follows from the result of Hardy and Kuijlaars, which we recall here. The general setup of \cite{HaKu12} involves the following ingredients.
\begin{enumerate}\itemsep=0pt
	\item[(a)] A sequence of closed subsets $\Delta_1, \ldots, \Delta_d$ of the 	complex plane, each with positive logarithmic capacity.
	\item[(b)] For each $i = 1,\ldots, d$, a lower semi-continuous function 	$V_i \colon \Delta_i \to \mathbb R \cup \{+\infty\}$ that is f\/inite on a set of positive capacity.
	\item[(c)] A symmetric positive def\/inite interaction matrix $C = (c_{ij})$ of size $d \times d$.
	\item[(d)] A sequence of positive numbers $m_1, \ldots, m_d$, such that
\begin{gather} \label{eq:liminfVi}
	\liminf_{|x| \to \infty, \, x \in \Delta_i} \left( V_i(x) -
	\left(\sum_{j=1}^d c_{ij} m_j \right) \log \big(1 + |x|^2\big) \right) > -\infty
	\end{gather}
for every $i=1, \ldots, d$ for which $\Delta_i$ is unbounded.
\end{enumerate}

Associated with the above data is the energy functional
\begin{gather} \label{eq:JV}
	J(\mu_1, \ldots, \mu_d) = \sum_{i,j=1}^d c_{ij} I(\mu_i,\mu_j) + \sum_{i=1}^d \int V_i(x) d\mu_i(x).
	\end{gather}
The problem is to minimize $J$ over $\prod\limits_{j=1}^d \mathcal M_{m_j}(\Delta_j)$, or some subset of it. Here we use $\mathcal M_m(\Delta)$ to denote the set of positive Borel measures on $\Delta$ of total mass $m > 0$. Such a minimization problem is called a \textit{weakly admissible vector equilibrium problem}.

The functional~\eqref{eq:JV} is not well-def\/ined on all of $\prod\limits_{j=1}^d \mathcal M_{m_j}(\Delta_j)$, since there is a problem with measures having overlapping supports, and with measures having unbounded supports. The problem with overlapping supports was solved by Beckermann et al.~\cite{BeKaMaWi13}. To handle the situation with unbounded supports, a regularization of $J$ is introduced in \cite[Theorem~2.6]{HaKu12} that comes from mapping the measures to the Riemann sphere and redef\/ining $J$ accordingly. This procedure involves a modif\/ication of the mutual energy \eqref{eq:mutualenergy} and the corresponding logarithmic energy~\eqref{eq:logenergy} to
\begin{gather*} %\label{eq:sphericalenergy}
	\tilde{I}(\mu,\nu) = \iint \log \left( \frac{\sqrt{1+|x|^2} \sqrt{1+|y|^2}}{|x-y|} \right) d\mu(x) d\nu(y),
\end{gather*}
and $\tilde{I}(\mu) = \tilde{I}(\mu,\mu)$, respectively.

We may call $\tilde{I}(\mu)$ and $\tilde{I}(\mu,\nu)$ the spherical (mutual) energy. Then 	
\begin{gather} \label{eq:JV2}
	\tilde{J}(\mu_1, \ldots, \mu_d)
	= \begin{cases} \displaystyle \sum\limits_{i,j=1}^d c_{ij} \tilde{I}(\mu_i,\mu_j)
		+ \sum\limits_{i=1}^d \int \tilde{V}_i(x) d\mu_i(x),
			& \text{if all } \tilde{I}(\mu_i) < +\infty, \\
		+\infty, & \text{otherwise},
			\end{cases}
	\end{gather}
where
\begin{gather} \label{eq:tildeVi}
	\tilde{V}_i(x) = V_i(x) -
	\left(\sum_{j=1}^d c_{ij} m_j \right) \log\big(1+|x|^2\big)
	\end{gather}
is an extension of \eqref{eq:JV} since for vectors of measures with $I(\mu_i) < +\infty$ for all $i$, we have that $\tilde{J}(\mu_1, \ldots, \mu_d)
= J(\mu_1, \ldots, \mu_d)$. Note that~$\tilde{V}_i$ is bounded away from $-\infty$ on~$\Delta_i$ because of the lower-semicontinuity of $V_i$ and the growth condition \eqref{eq:liminfVi} at inf\/inity.

The functional \eqref{eq:JV2} is thus def\/ined on $\prod\limits_{j=1}^d \mathcal M_{m_j}(\Delta_j)$ with values in $\mathbb R \cup \{+\infty\}$ only. In addition, $\tilde{J}$ has compact sublevel sets $\tilde{J}^{-1}((-\infty, \alpha])$ for any real~$\alpha$ (which implies that it is lower semi-continuous), and it is strictly convex on the set where it is f\/inite. As a consequence, there is a unique minimizer of $\tilde{J}$ on
$\prod\limits_{j=1}^d \mathcal M_{m_j}(\Delta_j)$, see \cite[Corollary~2.7]{HaKu12}.

\subsection{Variational conditions}\label{section2.2}

We use $U^{\nu}(z) = \int \log \frac{1}{|z-y|} d\nu(y)$, $z \in \mathbb C$, to denote the logarithmic potential of a measure~$\nu$ and
\begin{gather*} \tilde{U}^{\nu}(z) = \int \log \left( \frac{\sqrt{1+|z|^2} \sqrt{1+|y|^2}}{|z-y|} \right) d\nu(y) \end{gather*}
to denote the spherical potential.

The variational conditions for the vector equilibrium problem were not discussed in~\cite{HaKu12}. The following result is standard for the case $d=1$, and its extension to $d \geq 2$ is not dif\/f\/icult.

\begin{Lemma} \label{lem21}
 Let $\vec{\mu}^*= (\mu_1^*, \ldots, \mu_d^*)$ be a vector of
	measures in $\prod\limits_{j=1}^d \mathcal M_{m_j}(\Delta_j)$ with $\tilde{I}(\mu_i^*) < +\infty$ 	for all $i$. Suppose that for some constants $\ell_i$,
	\begin{gather} \label{eq:varcond}
		2 \sum_{j=1}^d c_{ij} \tilde{U}^{\mu_j^*} + \tilde{V}_i
	\ 	\begin{cases} = \ell_i & \text{on } \supp(\mu_i^*), \\
		\geq \ell_i & \text{on } \Delta_i,
		\end{cases}
		\end{gather}
	for every $i=1, \ldots, d$.
	Then $\vec{\mu}$ is the minimizer of $\tilde{J}$ in $\prod\limits_{j=1}^d \mathcal M_{m_j}(\Delta_j)$.
\end{Lemma}

\begin{proof}
	The functional $\tilde{J}$ is strictly convex on the set where
	it is f\/inite. The strict convexity is due to the quadratic part
	$\tilde{J}_0(\vec{\mu}) = \sum\limits_{i,j=1}^d c_{ij} \tilde{I}(\mu_i,\mu_j)$
	in the functional $\tilde{J}$ and it comes down to
	\begin{gather} \label{eq:varcond1}
		\tilde{J}_0(\vec{\mu}) + \tilde{J}_0(\vec{\nu}) \geq 2 \sum_{i,j=1}^d c_{ij} I(\mu_i,\nu_j),
		\end{gather}
whenever $\vec{\mu}, \vec{\nu} \in \prod\limits_{j=1}^d \mathcal M_{m_j}(\Delta_j)$, with strict inequality if $\vec{\mu} \neq \vec{\nu}$ and both $\tilde{J}(\vec{\mu}) < +\infty$, $\tilde{J}(\vec{\nu}) < +\infty$.
	
Let $\vec{\mu}^*$ be as in the lemma. Then for any $\vec{\mu} \in \prod\limits_{j=1}^d \mathcal M_{m_j}(\Delta_j)$, we f\/ind by integrating \eqref{eq:varcond} with respect to $\mu_i$,
	\begin{gather*} \int \left(2 \sum_{j=1}^d c_{ij} \tilde{U}^{\mu_j^*} + \tilde{V}_i \right) d\mu_i
		\geq \ell_i m_i, \end{gather*}
and so by summing over $i$,
	\begin{gather} \label{eq:varcond2}
		2 \sum_{i,j=1}^d c_{ij} \tilde{I}(\mu_i, \mu_j^*) + \sum_{i=1}^d \int \tilde{V}_i d\mu_i
		\geq \sum_{i=1}^d \ell_i m_i.
		\end{gather}
	
	If we integrate \eqref{eq:varcond} with respect to $\mu_i^*$ and sum over $i$,
	we f\/ind an equality
	\begin{gather*} 2 \sum_{i,j=1}^d c_{ij} \tilde{I}(\mu_i^*, \mu_j^*) + \sum_{i=1}^d \int \tilde{V}_i d\mu_i^*
	= \sum_{i=1}^d \ell_i m_i, \end{gather*}
which means{\samepage
	\begin{gather} \label{eq:varcond3}
		\sum_{i=1}^d \ell_i m_i = \tilde{J}(\vec{\mu}^*) + \tilde{J}_0(\vec{\mu}^*),
		\end{gather}
and in particular $\tilde{J}(\vec{\mu}^*) < +\infty$.}
		
Now we use \eqref{eq:varcond1} with $\vec{\nu} = \vec{\mu}^*$, and combine it with \eqref{eq:varcond2} and \eqref{eq:varcond3} to f\/ind
	\begin{align*}
	 \tilde{J}(\vec{\mu}) & = \tilde{J}_0(\vec{\mu}) + \sum_{i=1}^d \int \tilde{V}_i d\mu_i   \geq
	 2 \int \sum_{i,j=1}^d c_{ij} I(\mu_i,\mu_j^*) - \tilde{J}_0(\vec{\mu}^*)
		 + \sum_{i=1}^d \int \tilde{V}_i d\mu_i \\
	 & \geq \sum_{i=1}^d \ell_i m_i - \tilde{J}_0(\vec{\mu}^*)   = \tilde{J}(\vec{\mu}^*).
	\end{align*}
Thus $\vec{\mu}^*$ is indeed the minimizer of $\tilde{J}$.	
\end{proof}

\subsection{Nikishin interaction and balayage}\label{section2.3}

In the present paper we are dealing with the interaction matrix $C = (c_{ij})$ where
\begin{gather} \label{eq:Nikishin}
	c_{ij} = \begin{cases} 1, & \text{if } i = j, \\
	-\frac{1}{2}, & \text{if } |i-j| = 1, \\
	0, & \text{otherwise}, \end{cases}
	\end{gather}
which is indeed a positive def\/inite matrix, and sets $\Delta_j$ that alternate between the positive and negative real axis as in~\eqref{eq:Deltaj}. The interaction matrix \eqref{eq:Nikishin} is characteristic for Nikishin systems in the theory of Hermite--Pad\'e approximation~\cite{NikSor91}. See also \cite{AptKui11, Kui10} for surveys on the connections with Hermite-Pad\'e approximation and random matrix theory.

In Theorem \ref{thm12} we have $d=q+r-1$ and we use indices $i=-q+1, \ldots, r-1$, instead of $i=1,\ldots, d$. The total masses $m_j$ are as in
condition~ii) of Theorem~\ref{thm12}. The external f\/ields are $V_0 = \widehat{V}$ and $V_i \equiv 0$ for $i \neq 0$. Then it is easy to see that the conditions (a), (b), (c), and (d) stated in Section~\ref{section2.1} are satisf\/ied. In fact we have~\eqref{eq:liminfVi}, since
\begin{gather*} \label{eq:masssums}
	\sum_{j=-q+1}^{r-1} c_{ij} m_j = -\frac{1}{2} m_{i-1} + m_i - \frac{1}{2} m_{i+1}
	 = \begin{cases} 0, & \text{for } i \neq 0, \\
	 \frac{q^{-1} + r^{-1}}{2}, & \text{for } i = 0,
	 \end{cases} \end{gather*}
and then \eqref{eq:liminfVi} is satisf\/ied because of~\eqref{eq:tildeVgrowth}.

Hence there exists a unique minimizing vector of measures
$(\nu_{-q+1}^*, \ldots, \nu_{r-1}^*)$ for the energy functional
$\tilde{J}$ with
\begin{gather*} \tilde{V}(x) = \widehat{V}(x) - \frac{q^{-1} + r^{-1}}{2} \log\big(1+|x|^2\big), \end{gather*}
see \eqref{eq:tildeVi}. We have to show that in addition $I(\nu_i^*) < +\infty$ for all~$i$, and then we can conclude that it is also a minimizing vector for $J$.

The variational conditions \eqref{eq:varcond} from Lemma \ref{lem21} are in this case
\begin{gather} \label{eq:varcond4}
	2 \tilde{U}^{\nu_0^*} - \tilde{U}^{\nu_{-1}^*} - \tilde{U}^{\nu_{1}^*}
	+ \tilde{V} \ \begin{cases} = \ell_0 & \text{on } \supp(\nu_0^*), \\
		\geq \ell_0 & \text{on } [0,\infty), \end{cases}
		\end{gather}
and for $j \in \{-q+1, \ldots, r-1\} \setminus \{0\}$,{\samepage
\begin{gather} \label{eq:varcond5}
	2 \tilde{U}^{\nu_j^*} - \tilde{U}^{\nu_{j-1}^*} - \tilde{U}^{\nu_{j+1}^*}
			\ \begin{cases} = \ell_j & \text{on } \supp(\nu_j^*), \\
				\geq \ell_j & \text{on } \Delta_j. \end{cases}
			\end{gather}
Here we set $\nu_{-q}^* = \nu_{r}^*=0$ so that \eqref{eq:varcond4},~\eqref{eq:varcond5} also hold for $j=-q+1$ and $j=r-1$.}

Note that $2 \nu_j^*$ and $\nu_{j-1}^* + \nu_{j+1}^*$ have the same total masses if $j \neq 0$. If~\eqref{eq:varcond5} holds then $2\nu_j^*$ is the
\textit{balayage measure} of $\nu_{j-1}^* + \nu_{j+1}^*$ onto $\Delta_j$. The balayage measure has full support $\Delta_j$, and equality holds in~\eqref{eq:varcond5} everywhere on~$\Delta_j$. In addition, the constant $\ell_j$ is zero. So for $j \neq 0$ the variational condition is
\begin{gather} \label{eq:varcond6}
2 \tilde{U}^{\nu_j^*} - \tilde{U}^{\nu_{j-1}^*} - \tilde{U}^{\nu_{j+1}^*} = 0 \qquad \text{on} \ \ \Delta_j.
	\end{gather}
	
If $x \mapsto \log(1+|x|^2)$ is integrable with respect to all measures, then \eqref{eq:varcond4} reduces to
\begin{gather} \label{eq:varcond8}
2 U^{\nu_0^*} - U^{\nu_{-1}^*} - U^{\nu_{1}^*}
+ \widehat{V} \ \begin{cases} = \ell & \text{on } \supp(\nu_0^*), \\
\geq \ell & \text{on } [0,\infty), \end{cases}
\end{gather}
for some constant $\ell$, while \eqref{eq:varcond6} reduces to
\begin{gather} \label{eq:varcond7} 2 U^{\nu_j^*} = U^{\nu_{j-1}^*} + U^{\nu_{j+1}^*}
	\qquad \text{on} \ \ \Delta_j,
	\end{gather}
which is the more common form for the balayage in logarithmic potential theory, see \cite{SaTo97} where the discussion however is restricted to measures with compact support in~$\mathbb C$.

Our strategy to prove Theorem \ref{thm12} will be to establish the existence of a vector of measures $\vec{\nu}^* = (\nu_{-q+1}^*, \ldots, \nu_{r-1}^*)$ with $\supp(\nu_j^*) \subset \Delta_j$, $\nu_j^*(\Delta_j) = m_j$, such that the conditions~\eqref{eq:varcond7} and~\eqref{eq:varcond8} are satisf\/ied. The measure $\nu_0^*$ will have compact support, and all other measures have full support. The density of $\nu_j^*$ will behave like
\begin{gather*} \frac{d\nu_j^*(x)}{dx} =
	\begin{cases}
		O\big(|x|^{-1-1/q}\big) & \text{for } j \geq 1, \\
		O\big(|x|^{-1-1/r}\big) & \text{for } j \leq -1,
	\end{cases} \qquad \text{as} \ \ |x| \to \infty. \end{gather*}

\section{Proof of Theorem \ref{thm12}}\label{section3}

\subsection{An auxiliary result}\label{section3.1}

The following is our main auxiliary result.

\begin{Proposition} \label{prop21}
	Let $r \geq 2$ be an integer and let $a > 0$. Then there is a unique vector $(\mu_1^*, \ldots, \mu_{r-1}^*)$ of measures that minimizes the energy functional
	\begin{gather} \label{eq:CRfunctional3}
	\sum_{j=1}^{r-1} I(\mu_j) - \sum_{j=1}^{r-2} I(\mu_j, \mu_{j+1}) + \int \log |x-a| d \mu_1(x)
	\end{gather}
	among all vectors of measures satisfying for every $j=1, \ldots, r-1$,
	\begin{enumerate}\itemsep=0pt
		\item[$i)$] $\supp(\mu_j) \subset \Delta_j$,
		\item[$ii)$] $\mu_j(\Delta_j) = 1 - \frac{j}{r}$, and
		\item[$iii)$] $I(\mu_j) < +\infty$.
	\end{enumerate}
	
	Moreover,
	\begin{gather} \label{eq:CRfunctional3minimum}
	\sum_{j=1}^{r-1} I(\mu_j^*) - \sum_{j=1}^{r-2} I(\mu_j^*, \mu_{j+1}^*) 	= - \frac{1}{2} \int \log|x-a| d\mu_1^*(x)
	\end{gather}
	and
	\begin{gather} \label{eq:Umu1}
		 U^{\mu_1^*}(z) = \log \left| \frac{z^{1/r} - a^{1/r}}{z-a} \right|,
	\qquad z \in \mathbb C. \end{gather}
\end{Proposition}
\begin{proof}
The minimization of \eqref{eq:CRfunctional3} under the conditions i), ii), iii) is a weakly admissible vector equilibrium problem for $r-1$ measures with total masses $m_j = 1 - \frac{j}{r}$, $j=1, \ldots, r-1$. We set $\mu_0^* = \delta_a$, the Dirac point mass at $a > 0$ and $m_0 = 1$. Also $\mu_r^* = 0$ and $m_r = 0$. Then by the discussion in Section~\ref{section2}, there is a unique minimizer $(\mu_1^*, \ldots, \mu_{r-1}^*)$ for the extended functional. We are going to construct the measures $\mu_j^*$ explicitly. We show that these measures have densities with respect to Lebesgue measures that decay as $|x|^{-1 - 1/r}$ as $|x| \to \infty$, and that for $j=1, \ldots, r-1$,
	\begin{gather*} 2 U^{\mu_j^*} = U^{\mu_{j-1}^*} + U^{\mu_{j+1}^*}
	\qquad \text{on} \ \ \Delta_j. \end{gather*}
	This implies that also also
	\begin{gather*} 2 \tilde{U}^{\mu_j^*} = \tilde{U}^{\mu_{j-1}^*} + \tilde{U}^{\mu_{j+1}^*}
	\qquad \text{on} \ \ \Delta_j. \end{gather*}
Then by Lemma~\ref{lem21} it follows that $(\mu_1^*, \ldots, \mu_{r-1}^*)$ is the minimizer for the extended functional, and since the function $x \mapsto \log(1+|x|^2)$ is integrable with respect to each of the measures, it then also follows that it is the minimizer of~\eqref{eq:CRfunctional3}. We will see at the end of the proof that~\eqref{eq:CRfunctional3minimum} and~\eqref{eq:Umu1} hold as well.
	
We use a geometric construction based on the Riemann surface for the mapping $w = z^{1/r}$. The Riemann surface has $r$ sheets
	\begin{gather*}
		\mathcal R_j =
			\begin{cases} \mathbb C \setminus \Delta_1 = \mathbb C \setminus (-\infty,0],
				& \text{for } j = 1, \\
				\mathbb C \setminus (\Delta_{j-1} \cup \Delta_j) = \mathbb C \setminus \mathbb R, & \text{for } j=2,\ldots, r-1, \\
				\mathbb C \setminus \Delta_{r-1}, & \text{for } j= r,
				\end{cases}
	\end{gather*}
where $\mathcal R_j$ is connected to $\mathcal R_{j+1}$ along $\Delta_j$ for $j=1, \ldots, r-1$ in the usual crosswise manner. There is one point at inf\/inity that connects all $r$ sheets. Note that $\mathcal R_r = \mathbb C \setminus (-\infty,0]$ if $r$ is even, and $\mathcal R_r = \mathbb C \setminus [0,\infty)$ if $r$ is odd.
	
The Riemann surface has genus zero and $z = w^r$ is a rational parametrization of it. The rational function
	\begin{gather} \label{eq:Psi}
		\Psi(w) = \frac{1}{r w^{r-1}(w-a^{1/r})}, \qquad z = w^r,
		\end{gather}
is meromorphic on the Riemann surface with a simple pole at $z = a$, a pole of order $r-1$ at $z=0$, and a zero of order $r$ at $z=\infty$. We use $\Psi_j$ to denote its restriction to $\mathcal R_j$. Explicitly, we then have
	\begin{gather} \label{eq:Psi1}
		\Psi_1(z) = \frac{1}{r z^{1-1/r}(z^{1/r}-a^{1/r})},
			\qquad z \in \mathcal R_1 = \mathbb C \setminus (-\infty,0]
		\end{gather}
with the principal branch of the $r$th roots. On the other sheets we have the same formula but with dif\/ferent choices of $r$th roots in $z^{1-1/r}$ and $z^{1/r}$. We always use $a^{1/r} > 0$.
	
The cuts $\Delta_j$ are oriented from left to right, and we use $\Psi_{j,\pm}(s)$, $s \in \Delta_j$, to denote the limit of $\Psi_j(z)$ as $z \to s$ with $ \pm \operatorname{Im} z > 0$. Then we def\/ine for $j=1, \ldots, r-1$,
	\begin{gather} \label{eq:mujdensity}
		d\mu_j^*(s) = \frac{1}{2\pi i} (\Psi_{j,+}(s) - \Psi_{j,-}(s)) ds, \qquad s \in \Delta_j.
\end{gather}
This def\/ines a real measure on $\Delta_j$ since $\Psi_{j,-}(s) = \overline{\Psi_{j,+}(s)}$ for $s \in \Delta_j$, but a~priori it could be a~signed measure. Suppose the density vanishes at an interior point $s \in \Delta_j$. Then $\Psi_{j,\pm}(s)$ is real, which implies that $\Psi(w)$ is real for some $w \in \mathbb C$ with $w^r = s$. From the formula~\eqref{eq:Psi} for~$\Psi$ it then easily follows that~$w$ is real. However, if~$w$ would be real and positive then $w^r$ would be on~$[0,\infty)$ on the f\/irst sheet, and if~$w$ would be real and negative then $w^r$ would be on $\mathbb R \setminus \Delta_{r-1}$ on the $r$th sheet. Since $w^r = s$ is on one of the cuts, we have a contradiction and we see that the density~\eqref{eq:mujdensity} does not vanish at an interior, and therefore has a constant sign.
	
We compute the total masses by contour integration as in~\cite{DuKu08}. We consider $j=1$ f\/irst. Then by~\eqref{eq:mujdensity}
\begin{gather*} \int d\mu_j^*(s) = \frac{1}{2 \pi i}\int_{C} \Psi_1(s) ds, \end{gather*}
where $C$ is a contour that starts at $-\infty$ and follows the upper side of the cut $\Delta_1 = (-\infty,0]$, and goes back to $-\infty$ on the lower side of the cut. We deform the contour to a big circle $|s| = R$, and we pick up a residue condition from the pole at $s=a$. From~\eqref{eq:Psi1} we calculate the residue as
\begin{gather*} \lim_{z \to a} (z-a) \Psi_1(z) = \frac{1}{r a^{1-1/r}} \lim_{z \to a} \frac{z-a}{z^{1/r} - a^{1/r}} = 1.\end{gather*}
Hence
	\begin{gather*} \int d\mu_1^*(s) = 1 - \lim_{R\to \infty} \frac{1}{2\pi i} \int_{|s|=R}
		\Psi_1(s) ds \end{gather*}
with the circle $|s|=R$ oriented counterclockwise. Since for every $j=1, \ldots, r-1$,
	\begin{gather} \label{eq:Psiasymp}
		\Psi_j(z) = \frac{1}{rz} + O\big(z^{-1-1/r}\big) \qquad \text{as} \ \ z \to \infty,
		\end{gather}
which easily follows from \eqref{eq:Psi}, we f\/ind
	\begin{gather} \label{eq:mu1mass} \int d\mu_1^*(s) = 1 - \frac{1}{r}.
	\end{gather}
	
Now consider $2 \leq j \leq r-1$. Then by \eqref{eq:mujdensity} and the fact that $\Psi_{j-1,\pm} = \Psi_{j,\mp}$ on~$\Delta_{j-1}$,
	\begin{gather*} \int d\mu_{j-1}^*(s) - \int d\mu_{j}^*(s)
		= \frac{1}{2\pi i} \int_{\Delta_{j-1}} \left(-\Psi_{j,-}(s) - \Psi_{j,+}(s)\right) ds\\
\hphantom{\int d\mu_{j-1}^*(s) - \int d\mu_{j}^*(s) =}{}
- \frac{1}{2\pi i} \int_{\Delta_j} \left(\Psi_{j,+}(s)- \Psi_{j,-}(s) \right) ds \\
\hphantom{\int d\mu_{j-1}^*(s) - \int d\mu_{j}^*(s)}{} =
		\frac{1}{2\pi i} \int_{\mathbb R} \left(\Psi_{j,-}(s) - \Psi_{j,+}(s)\right) ds,
		\end{gather*}
since $\Delta_{j-1} \cup \Delta_j = \mathbb R$. Again by contour deformation this is
	\begin{gather*} \int d\mu_{j-1}^*(s) - \int d\mu_{j}^*(s) = \lim_{R\to \infty} \frac{1}{2\pi i} \int_{|s|=R} \Psi_j(s) ds = \frac{1}{r}, \end{gather*}
where we used \eqref{eq:Psiasymp}. Together with \eqref{eq:mu1mass} we conclude that
	\begin{gather*}	%\label{eq:mujmass}
\int d\mu_j^*(s) = 1 - \frac{j}{r}, \qquad j =1, \ldots, r-1.
	\end{gather*}
Since the total masses are positive, and the densities of the measures do not change sign, it now also follows that the measures are positive.
	
We introduce the Cauchy transforms of the measures
\begin{gather*} %\label{eq:defFj}
		F_j(z) = \int \frac{d\mu_j^*(s)}{z-s}, \qquad z \in \mathbb C \setminus \Delta_j.
\end{gather*}
Then by a similar contour integration argument, where now we pick up a residue contribution at $s=z$, while there is no contribution from inf\/inity, we get
\begin{gather}
F_1(z) = \frac{1}{2\pi i} \int_C \frac{\Psi_1(s)}{z-s} ds = \frac{1}{z-a} - \Psi_1(z), \qquad z \in \mathbb C \setminus (-\infty,0] \label{eq:F1z}
\end{gather}
	and for $j=2, \ldots, r-1$,
\begin{gather} \label{eq:Fjz}
		F_{j-1}(z) - F_j(z) = \frac{1}{2\pi i} \int_{\mathbb R} \frac{\Psi_{j,-}(s) - \Psi_{j,+}(s)}{z-s} ds = \Psi_j(z),\qquad z \in \mathbb C \setminus \mathbb R,
\end{gather}
	where $F_r(z) = 0$.
	The identity $\Psi_{j,+} = \Psi_{j+1,-}$ on $\Delta_j$ then leads to
	\begin{gather} \label{eq:Fjbalayage}
	 F_{j,+}(x) + F_{j,-}(x) = F_{j-1}(x) + F_{j+1}(x), \qquad x \in \Delta_j,
	\end{gather}	
for $j=2, \ldots, r-1$. By \eqref{eq:F1z} and \eqref{eq:Fjz}, the identity \eqref{eq:Fjbalayage} also holds for $j=1$, if we agree that
	\begin{gather*} F_0(z) = \frac{1}{z-a}. \end{gather*}
		
The measures have a density that decays like $|s|^{-1-1/r}$ as $|s| \to \infty$. This easily follows from the def\/initions~\eqref{eq:Psi} and~\eqref{eq:mujdensity}. Then $s \mapsto \log(1+s^2)$ is integrable for these measures, and the usual logarithmic potentials exist. By Sokhotskii--Plemelj formulas we have
	\begin{gather*} 2 \frac{d}{dx} U^{\mu_j^*}(x) = F_{j,+}(x) + F_{j,-}(x), \qquad x \in \Delta_j. \end{gather*}
Clearly also
	\begin{gather*} \frac{d}{dx} U^{\mu_{j\pm 1}^*}(x) = F_{j\pm 1}(x), \qquad x \in \Delta_j. \end{gather*}
Then by integrating \eqref{eq:Fjbalayage} we obtain
	\begin{gather} \label{eq:Ujbalayage}
	2 U^{\mu_j^*}(x) = U^{\mu_{j-1}^*}(x) + U^{\mu_{j+1}^*}(x), \qquad x \in \Delta_j.
	\end{gather}		
There is no constant of integration in \eqref{eq:Ujbalayage} since
	\begin{gather*} U^{\mu_i^*}(x) = (1-\tfrac{i}{r}) \log |x| + o(1) \qquad \text{as} \ \ |x| \to \infty, \end{gather*}
for each $i \in \{j-1, j,j+1\}$ and $\Delta_j$ is unbounded.
	
Thus we have reached the identity \eqref{eq:Ujbalayage} that we aimed for, as discussed in the beginning of the proof. It remains to verify~\eqref{eq:CRfunctional3minimum} and~\eqref{eq:Umu1}.
	
By \eqref{eq:Ujbalayage} we obtain for $j=1, \ldots, r-1$,
	\begin{gather*}
		I(\mu_j^*) = \int U^{\mu_j^*} d\mu_j^* =
			\frac{1}{2} \int \big( U^{\mu_{j-1}^*} + U^{\mu_{j+1}^*} \big) d\mu_j^*
		 = \frac{1}{2} \big(I(\mu_{j-1}^*, \mu_j^*) + I(\mu_j^*, \mu_{j+1}^*) \big),
	\end{gather*}
which implies
	\begin{gather*}
		\sum_{j=1}^{r-1} I(\mu_j^*) = \sum_{j=1}^{r-2} I(\mu_j^*, \mu_{j+1}^*) + \frac{1}{2}( I(\mu_0^*, \mu_1^*) + I(\mu_{r-1}^*, \mu_r^*) ).
	\end{gather*}
We recall that $\mu_0^* = \delta_a$ and $\mu_r^* = 0$ and we obtain{\samepage
	\begin{gather*}
	\sum_{j=1}^{r-1} I(\mu_j^*) - \sum_{j=1}^{r-2} I(\mu_j^*, \mu_{j+1}^*) = \frac{1}{2} I(\delta_a, \mu_1^*)
		= \frac{1}{2} \int \log \frac{1}{|x-a|} d\mu_1^*(x),
	\end{gather*}
which is the identity in \eqref{eq:CRfunctional3minimum}.}
	
Finally, we recall that by 	\eqref{eq:F1z} and \eqref{eq:Psi1}
	\begin{gather*} F_1(z) = \frac{1}{z-a} - \frac{1}{rz^{1-1/r}(z^{1/r} - a^{1/r})}, \qquad
		z \in \mathbb C \setminus (-\infty,0], \end{gather*}
which after integration leads to
	\begin{gather*} \int \log(z-s) d\mu_1^*(s) = \log(z-a) - \log\big(z^{1/r} - a^{1/r}\big). \end{gather*}
The constant of integration vanishes since both sides behave like $(1-1/r) \log z + o(1)$ as $z \to \infty$. Taking real parts we f\/ind~\eqref{eq:Umu1}.
	\end{proof}

We next extend Proposition \ref{prop21} from point masses $\delta_a$ with $a > 0$ to general measures with compact support on $(0,\infty)$.

\begin{Proposition} \label{prop22}
Let $r \geq 2$ be an integer. Let $\mu$ be a probability measure on $(0,\infty)$ with compact support. Then there is a unique vector $(\mu_1^*, \ldots, \mu_{r-1}^*)$ of measures that minimizes the energy functional
\begin{gather*} %\label{eq:CRfunctional4}
	 \sum_{j=1}^{r-1} I(\mu_j) - \sum_{j=1}^{r-2} I(\mu_j, \mu_{j+1}) - I(\mu,\mu_1)
\end{gather*}
among all vectors of measures satisfying for every $j=1, \ldots, r-1$,
\begin{enumerate}\itemsep=0pt
	\item[$i)$] $\supp(\mu_j) \subset \Delta_j$,
	\item[$ii)$] $\mu_j(\Delta_j) = 1 - \frac{j}{r}$, and
	\item[$iii)$] $I(\mu_j) < +\infty$.
\end{enumerate}

Moreover,
\begin{gather}
\sum_{j=1}^{r-1} I(\mu_j^*) - \sum_{j=1}^{r-2} I(\mu_j^*, \mu_{j+1}^*) - I(\mu,\mu_1^*)\nonumber\\
\qquad{} = - \frac{1}{2} I(\mu,\mu_1^*) = - \frac{1}{2} \iint \log \left| \frac{x^{1/r} - y^{1/r}}{x-y} \right| d \mu(x) d\mu(y).\label{eq:CRfunctional4minimum}
\end{gather}
\end{Proposition}
\begin{proof}
For $\mu = \delta_a$ this was done in Proposition \ref{prop21}.

Let $(\mu_1^*(a), \ldots, \mu_{r-1}^*(a))$ be the vector of measures that we obtain from $\delta_a$ as in Proposition~\ref{prop21}. Then for a general probability measure~$\mu$ on $(0,\infty)$ with compact support, we put
\begin{gather*} \mu_j^* = \int \mu_j^*(a) d\mu(a), \qquad j = 1, \ldots, r-1. \end{gather*}
These are well-def\/ined positive measures satisfying i), ii) and iii) of the proposition. The measures $\mu_j^*(a)$ have a density that decays
as $|x|^{-1-1/r}$ as $|x| \to \infty$, and the same will be true for the measures $\mu_j^*$ since $\mu$ is compactly supported. Thus the logarithmic potentials exist, and
	\begin{gather*} U^{\mu_j^*} = \int U^{\mu_j^*(a)} d\mu(a). \end{gather*}
The identity
\begin{gather*} 2 U^{\mu_j^*(a)} = U^{\mu_{j-1}^*(a)} + U^{\mu_{j+1}^*(a)} \qquad \text{on} \ \ \Delta_j \end{gather*}
holds for every $a > 0$ by Proposition~\ref{prop21}. Integrating this with respect to $a$ and using Fubini's theorem, we obtain
\begin{gather*} 2 U^{\mu_j^*} = U^{\mu_{j-1}^*} + U^{\mu_{j+1}^*} \qquad \text{on} \ \ \Delta_j. \end{gather*}
As in the proof of Proposition~\ref{prop21} this leads to the f\/irst identity of~\eqref{eq:CRfunctional4minimum}.

For $j=1$ we get
\begin{gather*} U^{\mu_1^*}(z) = \int U^{\mu_1^*(a)}(z) d\mu(a)
	= \int \log \left| \frac{z^{1/r} - a^{1/r}}{z-a} \right| d\mu(a),
	\end{gather*}
see \eqref{eq:Umu1}. Changing $z$ and $a$ to $x$ and $y$, and integrating over $d\mu(x)$, we obtain the second identity of~\eqref{eq:CRfunctional4minimum}.
\end{proof}

\section{Proofs of Theorems \ref{thm11} and \ref{thm12}}\label{section4}

Theorem \ref{thm11} is the special case $q=1$ of Theorem~\ref{thm12} and so it suf\/f\/ices to prove Theorem~\ref{thm12}.

\begin{proof}
Let $\nu_0$ be a probability measure on $(0,\infty)$ with compact support, and let $(\nu_{-q+1}, \ldots, \nu_{-1}$, $\nu_0, \nu_1, \ldots, \nu_{r-1})$ be the minimizing vector of measures for~\eqref{eq:VEproblem2} under the assumptions i), ii) and~iii) of Theorem~\ref{thm12}, with $\nu_0$ f\/ixed.
		
Then it is easy to see that $(\nu_1, \ldots, \nu_{r-1})$ is the	minimizer for the vector energy problem of Proposition~\ref{prop22} with $\mu = \nu_0$,	and hence by~\eqref{eq:CRfunctional4minimum},
	\begin{gather} \label{eq:Inuplus}
	\sum_{j=1}^{r-1} I(\nu_j) - \sum_{j=1}^{r-2} I(\nu_j, \nu_{j+1}) - I(\nu,\nu_1)
	= - \frac{1}{2} \iint \log \left| \frac{x^{1/r} - y^{1/r}}{x-y} \right| d\nu_0(x) d\nu_0(y). \end{gather}
	
Similarly, $(\nu_{-1}, \ldots, \nu_{-q+1})$ is the minimizer for the vector energy problem of Proposition~\ref{prop22} with $\mu= \nu_0$ and $q$ instead of~$r$. Thus by \eqref{eq:CRfunctional4minimum} again,
	\begin{gather} \label{eq:Inuminus}
	\sum_{j=-q+1}^{-1}\! I(\nu_j) - \sum_{j=-q+1}^{-2} \! I(\nu_j, \nu_{j+1}) - I(\nu_{-1},\nu_0)
	= - \frac{1}{2} \iint \!\log \left| \frac{x^{1/q} - y^{1/q}}{x-y} \right| d\nu_0(x) d\nu_0(y).\!\!\! \end{gather}
The identity also holds in case $q=1$, since then $\nu_{-1}^*=0$ and both sides of \eqref{eq:Inuminus} are~$0$.
	
From \eqref{eq:Inuplus} and \eqref{eq:Inuminus} we f\/ind that for a f\/ixed $\nu_0$, the minimum of the energy functional~\eqref{eq:VEproblem2} taken over all $\nu_j$, for $j=-q+1, -1$, $j=1, \ldots, r-1$ satsifying items~i),~ii) of Theorem~\ref{thm12}, is equal to
	\begin{gather*}
		J(\nu_{-q+1}, \ldots, \nu_0, \ldots, \nu_{r-1}) =
		I(\nu_0) - \frac{1}{2} \iint \log \left| \frac{x^{1/q} - y^{1/q}}{x-y} \right| d\nu_0(x) d\nu_0(y) \\
\hphantom{J(\nu_{-q+1}, \ldots, \nu_0, \ldots, \nu_{r-1}) =}{}
 - \frac{1}{2} \iint \log \left| \frac{x^{1/r} - y^{1/r}}{x-y} \right| d\nu_0(x) d\nu_0(y)
		 + \int \tilde{V}(x) d\nu_0(x), \end{gather*}
which reduces to \eqref{eq:CRfunctional2} with $\nu_0$ instead of $\nu$. Thus the component $\nu_0^*$ of the minimizer for the vector energy \eqref{eq:VEproblem} is also the minimizer of~\eqref{eq:CRfunctional2} over probability measures on $[0,\infty)$.
	
We f\/inally prove that $\nu_0^*$ has compact support. Let $\rho = \nu_{-1}^* + \nu_1^*$. Then $\nu_0^*$ is the minimizer of{\samepage
		\begin{gather} \label{eq:nu0compact1}
		\tilde{I}(\nu) - \tilde{I}(\nu,\rho) + \int \tilde{V} d \nu
		\end{gather}
among probability measures $\nu$ on $[0,\infty)$.}

\looseness=-1		
If $x \mapsto \log(1+|x|^2)$ would be integrable with respect to $\nu_0^*$ then it would also be the minimizer~of
		\begin{gather}
		I(\nu) - \tilde{I}(\nu,\rho) + \int \big(\tilde{V}(x) + \log \big(1+|x|^2\big)\big) d \nu(x) \nonumber\\
\qquad{} = I(\nu) + \int \left( \widehat{V}(x) - \int \log \frac{\sqrt{1+s^2}}{x-s} d\rho(s) \right) d\nu,\label{eq:nu0compact2}
		\end{gather}
which is a usual minimization problem for one measure with an external f\/ield
		\begin{gather*} \widehat{V}(x) + \int \log \frac{x-s}{\sqrt{1+s^2}} d\rho(s) \end{gather*}
that is continuous on $(0,\infty)$ (since $\widehat{V}$ is continuous and $\rho$ is a measure on $(-\infty,0]$).
		
It is easy to see that $x-s > \sqrt{1+s^2}$ for $x > 1$ and $s < 0$. Thus $\int \log \frac{x-s}{\sqrt{1+s^2}} d\rho(s) > 0$ for $x > 1$, and it follows from \eqref{eq:tildeVgrowth} that
		\begin{gather*} \lim_{x \to \infty} \left( \widehat{V}(x) + \int \log \frac{x-s}{\sqrt{1+s^2}} d\rho(s) - \log\big(1+x^2\big) \right) = +\infty, \end{gather*}
		which guarantees that~\eqref{eq:nu0compact2} has a minimizer with compact support. This minimizer also minimizes \eqref{eq:nu0compact1} and thus coincides with $\nu_0^*$ which thus has compact support.
		
Theorem \ref{thm12} is now fully proved.
\end{proof}

\section{A f\/inal remark}\label{section5}

We consider the minimization problem for \eqref{eq:CRfunctional} with $\theta = 1/r$. From Theorem~\ref{thm11} we obtain the following result that gives conditions that guarantee that the Cauchy transform of the minimizing measure is an algebraic function.

\begin{Proposition} \label{prop51} Let $ \theta = 1/r$ be rational and suppose the external field $V\colon [0,\infty) \to \mathbb R$ is such that $V'$ is a rational function. Let $\mu^*$ be the probability measure that minimizes~\eqref{eq:CRfunctional} among all probability measures on $[0,\infty)$ and suppose that $\mu^*$ is supported on a finite union of intervals Then its Cauchy transform
\begin{gather*} F(z) = \int \frac{d\mu^*(s)}{z-s} \end{gather*}
is the solution of an algebraic equation of degree $r+1$. That is, there exist rational func\-tions~$c_j(z)$ for $j=0, \ldots, r$,
such that
\begin{gather*} %\label{algeq}
\sum_{j=0}^{r} c_j(z) F(z)^j + F(z)^{r+1}= 0, \qquad z \in \mathbb C. \end{gather*}
\end{Proposition}

\begin{proof}
We turn to the vector equilibrium problem from Theorem~\ref{thm11} and denote the minimizing vector by $(\mu_0^*, \ldots, \mu_{r-1}^*)$ with $\mu_0^* = \mu^*$. The variational conditions are
\begin{gather*}
	2 U^{\mu^*}(x) - U^{\mu_1^*}(x) + V(x)
	\ \begin{cases} = \ell, & \text{for } x \in
	\supp(\mu^*), \\
	\geq \ell, & \text{for } x \in [0,\infty), \end{cases}
\end{gather*}
and for $j=1, \ldots, r-1$,
\begin{gather*}
	2 U^{\mu_j^*}(x) - U^{\mu_{j-1}^*}(x) - U^{\mu_{j+1}^*}(x) = 0 \qquad
	 \text{for} \ \ x \in \Delta_j,
	\end{gather*}
	where $\mu_r^* = 0$. We write
\begin{gather*} F_j(z) = \int \frac{d\mu_j^*(s)}{z-s}, 	\qquad j = 1, \ldots, r, \end{gather*}
and we f\/ind by dif\/ferentiating the variational conditions
\begin{alignat}{3}
	& F_{0,+}'(x) + F_{0,-}'(x) = F_1(x) + V'(x), \qquad && x \in \supp(\mu^*), & \nonumber\\
	& F_{j,+}'(x) + F_{j,-}'(x) = F_{j-1}(x) + F_{j+1}(x), \qquad && x \in \Delta_j,& \label{eq:Fvarcond}
\end{alignat}
for $j=1,\ldots, r-1$.

We construct a Riemann surface $\mathcal R$ with $r+1$ sheets $\mathcal R_j$, $j=0, \ldots, r$ given by
\begin{gather*}
	\mathcal R_0 = \mathbb C \setminus \supp(\mu^*), \\
	\mathcal R_1 = \mathbb C \setminus ((-\infty,0] \cup \supp(\mu^*)), \\
	\mathcal R_j = \mathbb C \setminus (\Delta_{j-1} \cup \Delta_j), \qquad \text{for} \ \ j = 2, \ldots, r-1, \\
	\mathcal R_r = \mathbb C \setminus \Delta_{r-1},
\end{gather*}
where $\mathcal R_0$ is connected to $\mathcal R_1$ along $\supp(\mu^*)$ and $\mathcal R_j$ is connected to $\mathcal R_{j+1}$ along $\Delta_j$ for $j=1, \ldots, r-1$ in the usual crosswise manner. After adding points at inf\/inity we obtain a compact Riemann surface, since $\supp(\mu^*)$ consists of a f\/inite union of intervals.

We def\/ine a function $\Psi$ on $\mathcal R$ by specifying on each of the sheets
\begin{gather*} %\label{eq:defPsi}
 \Psi (z) = 	\begin{cases}
	V'(z) - F_0(z), & \text{for } z \in \mathcal R_0, \\
	F_{j-1}(z) - F_{j}(z), & \text{for } z \in \mathcal R_j, \quad j = 1, \ldots, r.
	\end{cases}
\end{gather*}
Then $\Psi$ is meromorphic on each of sheets (since $V'$ is a rational function). Moreover, the variational conditions~\eqref{eq:Fvarcond} tell us that $\Psi$ extends to a meromorphic function on the full Riemann surface~$\mathcal R$. Then also $V' - \Psi$ is a meromorphic function on~$\mathcal R$ which agrees with $F_0$ on the zero sheet. Therefore $F = F_0$ satisf\/ies an algebraic equation of degree $r+1$.
\end{proof}

Note that $\mu^*$ minimizes $I(\mu) + \int (V - U^{\mu_1^*}) d\mu$ among all probability measures $\mu$ on $[0,\infty)$, and the external f\/ield $V - U^{\mu_1^*}$ is real analytic on $(0,\infty)$. If it were also real analytic at~$0$, then it would follow from results in~\cite{DeKrMc99} that~$\mu^*$ is supported on a f\/inite union of intervals. Maybe the methods of~\cite{DeKrMc99} can be adapted to the present situation, and then the assumption in Proposition~\ref{prop51} about the f\/inite number of intervals would be unnecessary.

The Riemann surface in the proof of Proposition~\ref{prop51} has genus $0$ if and only if $\supp(\mu^*) = [0,a]$ for some $a > 0$. This is the case if $V(x)= x$, and for more general conditions see~\cite[Theorem~1.8]{ClaRom13}. We note that this Riemann surface also appears in the paper of Forrester, Liu and Zinn-Justin~\cite{FoLiZJ15}, see Fig.~1 in that paper.

\subsection*{Acknowledgements}

The author is supported by long term structural funding--Methusalem grant of the Flemish Go\-vernment, by KU Leuven Research Grant OT/12/073, by the Belgian Interuniversity Attraction Pole P07/18, and by FWO Flanders projects G.0934.13 and G.0864.16.

\pdfbookmark[1]{References}{ref}
\LastPageEnding

\end{document}